\documentclass[10pt]{article}

\usepackage{amsmath,amsthm,amsfonts,amssymb,mathtools}
\usepackage{enumitem}

\usepackage[margin=2.3cm,nohead]{geometry}

\usepackage{url}
\usepackage{xcolor}
\usepackage[normalem]{ulem}
\usepackage{hyperref}
\usepackage{graphicx}
\usepackage{grffile}
\usepackage{float}
\usepackage{caption}
\usepackage{subcaption}

\usepackage{booktabs}
\usepackage{multirow}

\usepackage{algorithm}
\usepackage{algpseudocode}

\usepackage[
  backend=biber,
  style=ieee,
  citestyle=numeric-comp,
  url=false,
  isbn=false,
  sortcites=true
]{biblatex}
\setlist[enumerate]{label=(\roman*), align=left}

\newtheorem{theorem}{Theorem}

\newtheorem{proposition}[theorem]{Proposition}
\newtheorem{definition}[theorem]{Definition}
\newtheorem{assumption}[theorem]{Assumption}
\newtheorem{remark}[theorem]{Remark}
\newtheorem{example}{Example}

\newcommand{\R}{\mathbb{R}}
\newcommand{\nw}{\mathcal{W}}
\DeclareMathOperator{\ext}{ext}
\DeclareMathOperator{\inte}{int}

\newcommand{\revise}[1]{#1}
\newcommand{\deleted}[1]{#1}

\algnewcommand{\Input}[1]{%
  \State \textbf{Input:} {\raggedright #1}%
}
\algnewcommand{\Initialize}[1]{%
  \State \textbf{Initialize:}
  \Statex \hspace*{\algorithmicindent}\parbox[t]{.8\linewidth}{\raggedright #1}
}
\algnewcommand{\Output}[1]{%
  \State \textbf{Output:} {\raggedright #1}%
}

\begin{document}
\title{On Parallel and Batch-Cutting Strategies for Norm-Minimization-Based Convex Vector Optimization}

\author{
Mohammed Alshahrani \thanks{Corresponding author. Department of Mathematics, King Fahd University of Petroleum \& Minerals, Dhahran, 31261, Saudi Arabia\\
Interdisciplinary Research Center for Smart Mobility and Logistics, King Fahd University of Petroleum \& Minerals, Dhahran, 31261, Saudi Arabia, (e-mail:{\tt mshahrani@kfupm.edu.sa}).}
}

\maketitle

\begin{abstract}
We develop parallel and batch-cutting variants of the norm-minimization-based outer approximation algorithm for convex vector optimization. The standard algorithm solves $N_k$ independent subproblems at each iteration~$k$ to evaluate all vertices of the current polyhedral approximation, but processes only the single best cut. We propose two improvements. First, we parallelize the \revise{subproblem evaluations} across $\nw$ workers, reducing per-iteration wall-clock time. Second, we introduce a batch-cutting strategy that adds up to $K$ supporting halfspaces per iteration, using information from all solved subproblems rather than discarding it. We prove that the batch-cutting variant inherits the convergence rate $O(k^{2/(1-q)})$ of the standard algorithm, where $k$ is the number of outer iterations and $q$ is the number of objectives. Computational experiments on eight test problems with $q \in \{2,3,4,5\}$ show that parallelism on 8 cores \revise{increases the speed by a factor of 1.1 to 4.2}, and batch cutting consistently reduces the iteration count by 62--80\%. However, the wall-clock benefit of batch cutting is problem-dependent: the additional cuts per iteration accelerate vertex count growth, so batch cutting is most effective when per-vertex subproblem cost dominates.
\end{abstract}


\medskip


\section{Introduction}\label{sec:intro}

The norm-minimization-based outer approximation algorithm of Ararat et al.~\cite{Ararat2022} is an established method for solving convex vector optimization problems (CVOPs) with bounded upper image \revise{(see Section~\ref{sec:prelim})} and $q$ objectives. At each iteration~$k$, the algorithm maintains a polyhedral outer approximation $P_k$ of the upper image and refines it by solving, for each vertex $v$ of $P_k$, the closest-point subproblem
\begin{equation}\label{eq:subproblem}
y^v := \arg\min_{y \in A} \|y - v\|, \qquad z^v := y^v - v,
\end{equation}
where $A$ is a compact convex slice of the upper image and $\|\cdot\|$ is a fixed norm on $\R^q$. The vertex $v^k$ maximizing $\|z^v\|$ is selected, and the supporting halfspace $H(w^k, A)$ defined by the cut normal $w^k := \nabla\|z^{v^k}\|$ is added: $P_{k+1} := P_k \cap H(w^k, A)$.

The convergence theory of this algorithm is well understood. Ararat et al.~\cite{Ararat2024} proved the Hausdorff approximation error satisfies $\delta_H(P_k, A) \mathrel{\revise{\in}} O(k^{2/(1-q)})$ for the Euclidean norm, matching the optimal rate for polyhedral approximation of smooth convex bodies~\cite{Gruber1993,Glasauer1997}. This rate was extended to fixed inner-product norms in~\cite{Alshahrani2026} and to all $\ell_p$ norms in~\cite{Alshahrani2026a}. Convergence under iteration-dependent adaptive metrics was also established in~\cite{Alshahrani2026}.

Yet the computational cost of this algorithm has received less attention. The bottleneck is the subproblem evaluation~\eqref{eq:subproblem}: at iteration $k$, \revise{the current outer approximation $P_k$ has $N_k$ vertices, and the algorithm solves the subproblem against each one, that is, $N_k$ independent convex programs}. Since $N_k$ grows with $k$ and with the dimension~$q$, the per-iteration cost dominates the overall runtime. On the other hand, the $N_k$ subproblems are independent---each depends only on the current polyhedron and the target set---so they can be solved in parallel. \revise{Although the standard algorithm caches the subproblem solutions for reuse at vertices that persist in later iterations~\cite[Algorithm~1, lines~11 and~17]{Ararat2024}, it adds only the single cut from the farthest vertex $v^k$ at each iteration, leaving the supporting halfspaces produced by the other solved subproblems unexploited.} Adding multiple cuts per iteration should reduce the total number of iterations needed.

In this paper we exploit both observations. We formalize a parallel version of the algorithm in which the $N_k$ \revise{subproblems} are distributed across $\nw$ workers (Algorithm~\ref{alg:parallel}), and we introduce a batch-cutting strategy (Algorithm~\ref{alg:batch}) that adds up to $K$ supporting halfspaces per iteration from the $K$ farthest vertices. We prove that batch cutting inherits the convergence rate $O(k^{2/(1-q)})$ of the standard algorithm (Theorem~\ref{thm:batch-rate}). The speedup comes from reducing the number of outer iterations. However, batch cutting also increases the vertex count of the outer approximation faster, since each batch adds $K$ new halfspaces instead of one. This increases the per-iteration cost at subsequent iterations, creating a tradeoff between iteration reduction and per-iteration overhead that we analyze experimentally in Section~\ref{sec:experiments}.

The outer approximation approach to CVOP was introduced by Benson~\cite{Benson1998} for multi-objective linear programming and extended to the convex case by Ehrgott et al.~\cite{Ehrgott2011} and L\"ohne et al.~\cite{Lohne2011,Lohne2014}. The norm-minimization variant~\cite{Ararat2022} eliminates the direction bias inherent in Pascoletti--Serafini scalarizations~\cite{Keskin2023}; convergence rates were established in~\cite{Ararat2024} using the $H$-sequence framework of Kamenev~\cite{Kamenev1992,Kamenev2002} and Lotov et al.~\cite{Lotov2004}.

Adding multiple cuts per iteration has a long history in scalar optimization, going back to the cutting plane methods of Kelley~\cite{KelleyJr.1960} and Cheney--Goldstein~\cite{Cheney1959}; multi-cut variants are standard in bundle methods~\cite{Hiriart-Urruty1993}. Our batch-cutting strategy brings this idea to vector optimization, \revise{where the cuts are generated at the vertices of the current outer approximation---one supporting halfspace of the upper image per vertex---rather than at successively queried points, as in the scalar cutting-plane methods above}.

Parallel algorithms for multi-objective optimization have been developed primarily in the evolutionary computation literature~\cite{Deb2002,Zhang2007}; parallel vertex enumeration was studied by Avis and Jordan~\cite{Avis2018}. To our knowledge, the present paper is the first to study parallelization and batch cutting for norm-minimization-based CVOP algorithms.

We validate the proposed algorithms on eight test problems from the CVOP literature~\cite{Ararat2024,Jahn2011,Keskin2023,VIENNET1996,Ansary2015}, spanning objective dimensions $q \in \{2,3,4,5\}$ and a range of constraint geometries. In addition to the parallel and batch techniques, we compare Euclidean and adaptive~\cite{Alshahrani2026} scalarization metrics and evaluate \revise{an LP-probe vertex strategy (Section~\ref{sec:design})} as an alternative to full enumeration\revise{, complementing prior vertex-selection rules~\cite{Keskin2023} that instead choose among the enumerated vertices}.

The remainder of the paper is organized as follows. Section~\ref{sec:prelim} recalls the convex vector optimization setting and the convergence rate framework. Section~\ref{sec:parallel} formalizes the parallel algorithm and derives the speedup model. Section~\ref{sec:batch} introduces the batch-cutting variant and proves convergence. Section~\ref{sec:experiments} presents the computational experiments. Section~\ref{sec:conclusion} concludes.

\section{Preliminaries}\label{sec:prelim}

Before presenting the parallel algorithms, we establish notation and recall the convergence rate framework from~\cite{Ararat2022,Ararat2024,Alshahrani2026a}. Throughout, $q \geq 2$ denotes the number of objectives.

We consider the convex vector optimization problem
\begin{equation}\label{P}
\tag{P}
\min \Gamma(x) \quad \text{w.r.t. } \leq_C \quad \text{s.t. } x \in X,
\end{equation}
where $C \subset \R^q$ is a closed, solid, pointed, polyhedral cone inducing the partial order $x \leq_C y \iff y - x \in C$, the set $X \subset \R^n$ is nonempty compact convex, and $\Gamma: X \to \R^q$ is continuous and $C$-convex (i.e., $\Gamma(\lambda x + (1-\lambda)x') \leq_C \lambda\,\Gamma(x) + (1-\lambda)\,\Gamma(x')$ for all $x, x' \in X$ and $\lambda \in [0,1]$). The upper image is $\mathcal{P} := \Gamma(X) + C$. For a convex set $B \subset \R^q$, we write $\ext(B)$ for its set of extreme points (vertices), and $\delta_H(B, B') := \max\bigl(\sup_{b \in B} \inf_{b' \in B'} \|b - b'\|,\; \sup_{b' \in B'} \inf_{b \in B} \|b - b'\|\bigr)$ for the Hausdorff distance between $B$ and $B'$. A supporting halfspace of a convex set $A$ with outer normal $w$ is $H(w, A) := \{y \in \R^q : w^T y \leq \sup_{a \in A} w^T a\}$.

\begin{assumption}[Standing hypotheses]\label{ass:standing}
Throughout this paper, we assume:
\begin{enumerate}
\item[(a)] $C \subset \R^q$ is a closed, solid, pointed, nontrivial, polyhedral cone.
\item[(b)] $X \subset \R^n$ is a nonempty compact convex set with $\inte X \neq \emptyset$.
\item[(c)] $\Gamma: X \to \R^q$ is continuous and $C$-convex.
\end{enumerate}
Under these conditions, we work with a compact convex slice $A$ of the upper image with $\inte A \neq \emptyset$, obtained by intersecting $\mathcal{P}$ with a halfspace as in~\cite{Ararat2022}.
\end{assumption}

We briefly recall the norm-minimization algorithm of~\cite{Ararat2022}; see~\cite{Ararat2024,Alshahrani2026,Alshahrani2026a} for detailed analyses. The algorithm maintains a sequence of polyhedral outer approximations $\{P_k\}_{k \geq 0}$ with $P_k \supseteq A$ for all $k$.

At iteration $k$:
\begin{enumerate}[label=\arabic*.]
\item Compute the vertex set $V_k := \ext(P_k)$, with $N_k := |V_k|$.
\item For each $v \in V_k$, solve the subproblem: $y^v := \arg\min_{y \in A} \|y - v\|$, and compute $z^v := y^v - v$ and $w^v := \nabla\|z^v\|$.
\item Select $v^k \mathrel{\revise{\in}} \arg\max_{v \in V_k} \|z^v\|$.
\item If $\|z^{v^k}\| \leq \varepsilon$, stop. Otherwise, set $P_{k+1} := P_k \cap H(w^k, A)$, where $w^k := w^{v^k}$.
\end{enumerate}

The computational cost of each iteration is dominated by Step~2: solving $N_k$ independent convex optimization problems. Step~1 (vertex enumeration) is comparatively cheap, accounting for less than 1\% of total runtime in our experiments.

We use the convergence rate framework of Kamenev~\cite{Kamenev1992} and Lotov et al.~\cite{Lotov2004}, as adapted to CVOP in~\cite{Ararat2024}.

\begin{definition}[$H(r,A)$-sequence of cutting~{\cite[Definition~8.3]{Lotov2004}}]\label{def:h-sequence}
Let $A \subset \R^q$ be a nonempty compact convex set and $(P_k)_{k \geq 0}$ \revise{be} a sequence of polytopes with $P_k \supseteq A$. We say $(P_k)$ is an $H(r,A)$-sequence for $r > 0$ if, for every $k \geq 0$, there exists a supporting halfspace $H_k = H(w^k, A)$ such that $P_{k+1} = P_k \cap H_k$ and
\begin{equation}\label{eq:h-property}
\delta_H(P_k, P_{k+1}) \geq r \cdot \delta_H(P_k, A).
\end{equation}
\end{definition}

\begin{proposition}[{\cite[Corollary~6.5]{Ararat2024}}]\label{prop:h-sequence}
Under Assumption~\ref{ass:standing}, the norm-minimization algorithm generates an $H(1,A)$-sequence with $\delta_H(P_k, P_{k+1}) = \delta_H(P_k, A) = \|z^{v^k}\|$.
\end{proposition}

\begin{theorem}[Improved convergence rate~{\revise{\cite[Theorem~7.2]{Ararat2024}},~\cite[Theorem~4.1]{Alshahrani2026a}}]\label{thm:rate}
Under Assumption~\ref{ass:standing}, the norm-minimization algorithm using $\|\cdot\|_p$ scalarization for any $p \in (1,\infty)$ satisfies
\begin{equation}\label{eq:rate}
\delta_H(P_k, A) \mathrel{\revise{\in}} O\bigl(k^{2/(1-q)}\bigr) \quad \text{as } k \to \infty.
\end{equation}
\end{theorem}

\section{Parallel algorithm}\label{sec:parallel}

As noted in Section~\ref{sec:prelim}, the computational cost of each iteration is dominated by Step~2: solving $N_k$ independent subproblems. Since these \revise{subproblems} do not depend on each other, they can be distributed across multiple workers. We now formalize this and analyze the resulting speedup.

\begin{algorithm}[H]
\caption{Parallel Norm-Minimization Algorithm}\label{alg:parallel}
\begin{algorithmic}[1]
\Input{Problem $(\Gamma, X, C)$, slice $A$, tolerance $\varepsilon > 0$, norm $\|\cdot\|$, workers $\nw$}
\Output{$\varepsilon$-approximate outer approximation $P_k$}
\Initialize{Compute initial outer approximation $P_0 \supseteq A$ (see~\cite{Ararat2022}), set $k \gets 0$}
\While{true}
  \State Compute vertex set $V_k \gets \ext(P_k)$, $N_k \gets |V_k|$ \Comment{Sequential}
  \State \textbf{parallel for} $v \in V_k$ \textbf{using} $\nw$ \textbf{workers:} \Comment{Parallel}
  \State \quad Solve $y^v \gets \arg\min_{y \in A} \|y - v\|$
  \State \quad Compute $z^v \gets y^v - v$, $w^v \gets \nabla\|z^v\|$
  \State \textbf{end parallel for}
  \State $v^k \gets \arg\max_{v \in V_k} \|z^v\|$ \Comment{Reduction}
  \If{$\|z^{v^k}\| \leq \varepsilon$}
    \State \textbf{return} $P_k$
  \EndIf
  \State $P_{k+1} \gets P_k \cap H(w^{v^k}, A)$, $k \gets k+1$
\EndWhile
\end{algorithmic}
\end{algorithm}

\revise{We refer to Algorithm~\ref{alg:parallel} as the \emph{synchronous} parallel algorithm: at each iteration the $N_k$ subproblems are distributed across the $\nw$ workers, and the algorithm waits for all of them to complete---a synchronization barrier at the end of the parallel-for loop---before the farthest vertex is selected and its cut is added. This synchronous parallel algorithm is the baseline against which batch cutting is compared in Section~\ref{sec:experiments}.}

\begin{proposition}[Convergence of parallel algorithm]\label{prop:parallel-convergence}
Algorithm~\ref{alg:parallel} generates the same $H(1,A)$-sequence as the sequential algorithm. In particular, the convergence rate~\eqref{eq:rate} holds.
\end{proposition}

\begin{proof}
At each iteration, the parallel for-loop computes the same values $\{y^v, z^v, w^v\}_{v \in V_k}$ as the sequential loop, only in parallel. The selection of $v^k$ and the cut $H(w^{v^k}, A)$ are identical, so the two algorithms produce the same sequence $\{P_k\}$. The claim follows from Proposition~\ref{prop:h-sequence} and Theorem~\ref{thm:rate}.
\end{proof}

Let $T_k^{\mathrm{vrep}}$ (vertex representation) and $T_k^{\mathrm{other}}$ denote the time for vertex enumeration and bookkeeping at iteration $k$, respectively. The sequential time per iteration is
\[
T_k^{\mathrm{seq}} = T_k^{\mathrm{vrep}} + N_k \cdot \bar{t}_k + T_k^{\mathrm{other}},
\]
where $\bar{t}_k$ is the average time per subproblem \deleted{solve}. With $\nw$ parallel workers:
\begin{equation}\label{eq:parallel-time}
T_k^{\mathrm{par}}(\nw) = T_k^{\mathrm{vrep}} + \left\lceil N_k / \nw \right\rceil \cdot \bar{t}_k + T_k^{\mathrm{overhead}} + T_k^{\mathrm{other}},
\end{equation}
where $T_k^{\mathrm{overhead}}$ accounts for communication and synchronization costs.

Define the \emph{serial fraction} at iteration $k$ as $f_k := (T_k^{\mathrm{vrep}} + T_k^{\mathrm{other}}) / T_k^{\mathrm{seq}}$. Then Amdahl's law~\cite{amdahlValiditySingleProcessor1967} gives the per-iteration speedup bound:
\begin{equation}\label{eq:amdahl}
S_k(\nw) := \frac{T_k^{\mathrm{seq}}}{T_k^{\mathrm{par}}(\nw)} \leq \frac{1}{f_k + (1 - f_k)/\nw}.
\end{equation}
Since $f_k < 0.01$ in our experiments, the theoretical speedup limit is close to $\nw$ for moderate $\nw$\revise{---that is, while $f_k(\nw - 1) \ll 1$, i.e., $\nw \ll 1 + 1/f_k \approx 100$, a range that comfortably includes the $\nw \leq 8$ workers used in our experiments}.

\begin{remark}[Scaling with $q$]\label{rem:scaling}
The vertex count $N_k$ grows with the objective dimension $q$. For fixed $\nw$, the ratio $N_k / \nw$ increases, improving parallel utilization and pushing the realized speedup closer to the Amdahl bound. \revise{This is broadly consistent with the speedups in Section~\ref{sec:experiments} (Table~\ref{tab:parallel}): the smallest speedup occurs on the problem with the fewest vertices per iteration (MOP7, where the speed increases by a factor of only $1.1$), while problems whose parallelizable subproblem work dominates the serial overhead approach the Amdahl bound (for example AP1, a factor of $4.2$, where expensive subproblem evaluations play the same role as a large vertex count).}
\end{remark}

\section{Batch-cutting algorithm}\label{sec:batch}

The parallel algorithm evaluates all $N_k$ vertices but uses only the single best cut. We now propose a variant that adds multiple cuts per iteration.

\begin{algorithm}[H]
\caption{Batch-Cutting Norm-Minimization Algorithm}\label{alg:batch}
\begin{algorithmic}[1]
\Input{Problem $(\Gamma, X, C)$, slice $A$, tolerance $\varepsilon > 0$, norm $\|\cdot\|$, workers $\nw$, max cuts $K$}
\Output{$\varepsilon$-approximate outer approximation $P_k$}
\Initialize{Compute initial outer approximation $P_0 \supseteq A$ (see~\cite{Ararat2022}), set $k \gets 0$}
\While{true}
  \State Compute vertex set $V_k \gets \ext(P_k)$, $N_k \gets |V_k|$
  \State \textbf{parallel for} $v \in V_k$ \textbf{using} $\nw$ \textbf{workers:}
  \State \quad Solve $y^v \gets \arg\min_{y \in A} \|y - v\|$
  \State \quad Compute $z^v \gets y^v - v$, $w^v \gets \nabla\|z^v\|$
  \State \textbf{end parallel for}
  \State $v^k \gets \arg\max_{v \in V_k} \|z^v\|$
  \If{$\|z^{v^k}\| \leq \varepsilon$}
    \State \textbf{return} $P_k$
  \EndIf
  \State Sort vertices by $\|z^v\|$ in decreasing order: $v^{(1)}, v^{(2)}, \ldots, v^{(N_k)}$ with $\|z^{v^{(1)}}\| \geq \|z^{v^{(2)}}\| \geq \cdots$
  \State Select cut set $\mathcal{C}_k := \{v^{(j)} : j = 1, \ldots, \min(K, \revise{N^\varepsilon_k})\}$
  \State $P_{k+1} \gets P_k \cap \bigcap_{v \in \mathcal{C}_k} H(w^v, A)$, $k \gets k+1$
\EndWhile
\end{algorithmic}
\end{algorithm}

At each iteration, Algorithm~\ref{alg:batch} adds up to $K$ supporting halfspaces simultaneously, one for each of the $K$ vertices farthest from $A$ (i.e., those with the largest $\|z^v\|$). \revise{Here $N^\varepsilon_k := |\{v \in \ext(P_k) : \|z^v\| > \varepsilon\}|$ denotes the number of vertices whose deviation exceeds the tolerance---the eligible cuts at iteration~$k$---so the batch adds $\min(K, N^\varepsilon_k)$ of them.} The cuts are valid since each $H(w^v, A)$ is a supporting halfspace of $A$ (by the optimality of the subproblem). When $K = 1$, Algorithm~\ref{alg:batch} reduces to Algorithm~\ref{alg:parallel}.

We show that batch cutting preserves the convergence rate of the standard algorithm. The idea is simple: each outer iteration of Algorithm~\ref{alg:batch} still includes \revise{the cut corresponding to the farthest vertex}---the same cut that the standard algorithm would make---and the additional cuts can only help.

\begin{theorem}[Convergence rate of batch cutting]\label{thm:batch-rate}
Under Assumption~\ref{ass:standing}, let $\|\cdot\| = \|\cdot\|_p$ for some $p \in (1,\infty)$. The sequence $\{P_k\}_{k \geq 0}$ generated by Algorithm~\ref{alg:batch} satisfies
\begin{equation}\label{eq:batch-rate}
\delta_H(P_k, A) \mathrel{\revise{\in}} O\bigl(k^{2/(1-q)}\bigr),
\end{equation}
where $k$ is the number of outer iterations.
\end{theorem}

\begin{proof}
The proof applies the four-step strategy of~\revise{\cite[Theorem~7.2]{Ararat2024}} to the farthest-vertex cuts produced by Algorithm~\ref{alg:batch}.

At each outer iteration $k$, Algorithm~\ref{alg:batch} selects the farthest vertex $v^{(1)}_k \in \arg\max_{v \in V_k} \|z^v\|$ and includes $H(w^{(1)}_k, A)$ as the first cut in the batch $\mathcal{C}_k$. By Proposition~\ref{prop:h-sequence}, this cut satisfies
\begin{equation}\label{eq:farthest-h}
\delta_H(P_k,\, P_k \cap H(w^{(1)}_k, A)) = \delta_H(P_k, A) =: h_k.
\end{equation}
Since $P_{k+1} = P_k \cap \bigcap_{v \in \mathcal{C}_k} H(w^v, A) \subseteq P_k \cap H(w^{(1)}_k, A)$, we have $h_{k+1} \leq h_k$, so $\{h_k\}$ is non-increasing.

\medskip
\noindent\emph{Step 1: Residual counting.}
By strict convexity of the norm, each farthest-vertex cut produces a new, distinct support point $y^{(1)}_k \in \partial A$ with cut normal $w^{(1)}_k$ \revise{(see the proof of~\cite[Lemma~7.7]{Ararat2024})}. After $k$ outer iterations, the algorithm has generated at least $k + J + 1$ distinct support points from farthest-vertex cuts, where $J + 1$ is the number of halfspaces defining $P_0$~\revise{\cite[Lemma~7.7]{Ararat2024}}.

\medskip
\noindent\emph{Step 2: Separation of deviation vectors.}
Let $\eta > 0$ denote the inradius of $A$ (the radius of the largest ball contained in $A$) and define the deviation vectors $\alpha_i := y^{(1)}_i - \eta\, w^{(1)}_i$ for $i = 0, \ldots, k-1$, corresponding to the farthest-vertex cut at each outer iteration. The separation bound~\revise{\cite[Lemma~7.5]{Ararat2024}} (see also~\cite[Theorem~4.1]{Alshahrani2026a}) applies: for all distinct pairs $i < j$,
\begin{equation}\label{eq:batch-separation}
\|\alpha_i - \alpha_j\| \geq \min\bigl\{C_2\, \eta,\; C_3 \sqrt{\eta \, h_{j-1}}\bigr\},
\end{equation}
\revise{where $C_2 = \sqrt{2}$ and $C_3 \in (0,1)$ are the explicit geometric constants of~\cite[Lemma~7.5]{Ararat2024}: the floor $C_2\eta$ comes from $\|w^{(1)}_i - w^{(1)}_j\|^2 = 2$ for the unit cut normals when $(w^{(1)}_i)^\top w^{(1)}_j \leq 0$, and the shrinking term $C_3\sqrt{\eta h}$ comes from the bound $\sqrt{\eta h} - h$, so any $C_3 \in (0,1)$ is admissible once $h$ is small. For a general $\ell_p$ norm the same constants are recovered up to the $\ell_p$--$\ell_2$ norm-equivalence factors~\cite[Theorem~4.1]{Alshahrani2026a}; in every case they depend only on the unit-ball geometry (hence the norm) and the dimension $q$, not on the problem instance.} The proof of~\eqref{eq:batch-separation} in~\revise{\cite[Lemma~7.5]{Ararat2024}} uses only three properties of the polytope sequence, all of which hold for the batch algorithm:
\begin{enumerate}
\item the farthest-vertex selection rule is the same as in the standard algorithm, so the optimality conditions and support-point geometry at each outer iteration are identical;
\item the polytopes are nested: $P_0 \supseteq P_1 \supseteq \cdots \supseteq A$, since each $P_{k+1}$ is obtained from $P_k$ by intersecting with valid supporting halfspaces of $A$;
\item each vertex $v^{(1)}_j$ of $P_j$ lies in all previously added halfspaces: since $P_j \subseteq P_i \cap H(w^{(1)}_i, A)$ for $i < j$ (by nesting and the fact that $H(w^{(1)}_i, A) \in \mathcal{C}_i$), we have $v^{(1)}_j \in H(w^{(1)}_i, A)$ for all $i < j$.
\end{enumerate}
\revise{Although the batch sequence $\{P_k\}$ is not an $H(r,A)$-sequence in the sense of Definition~\ref{def:h-sequence} (each outer step adds several halfspaces), the bound~\eqref{eq:batch-separation} uses only properties (i)--(iii) and not the single-cut structure, so it still applies. We extract a uniform separation as follows. Fix $k$ and take any pair $i < j \leq k$. Since $\{h_k\}$ is non-increasing and $j-1 \leq k-1$, we have $h_{j-1} \geq h_k$, hence $C_3\sqrt{\eta h_{j-1}} \geq C_3\sqrt{\eta h_k}$ and, by~\eqref{eq:batch-separation}, $\|\alpha_i - \alpha_j\| \geq \min\{C_2\eta,\, C_3\sqrt{\eta h_k}\}$. Because $h_k \to 0$---each outer iteration includes the farthest-vertex cut, which alone forces convergence---there is an index $K_0$ with $C_3\sqrt{\eta h_k} \leq C_2\eta$ for all $k \geq K_0$; for such $k$ the minimum equals $C_3\sqrt{\eta h_k}$, so the $k$ farthest-vertex deviation vectors $\alpha_0, \ldots, \alpha_{k-1}$ are pairwise $\varepsilon_k$-separated with $\varepsilon_k := C_3\sqrt{\eta h_k}$.}

\medskip
\noindent\emph{Step 3: Packing bound.}
The $k$ deviation vectors are $\varepsilon_k$-separated and lie in a ball of radius $R + \eta$, where $R := \sup_{a \in A} \|a\|$ bounds the size of $A$. \revise{Because the support points $y^{(1)}_i \in \partial A$ that generate these vectors lie on the $(q-1)$-dimensional boundary of $A$, the pairwise $\varepsilon_k$-separated deviation vectors obey the boundary packing (dispersion) estimate that underlies the optimal $n^{-2/(q-1)}$ polytopal approximation order; see Gruber~\cite[ineq.~(1.1)--(1.2)]{Gruber1993} and Glasauer--Gruber~\cite[Thm.~2]{Glasauer1997}. The corresponding inequality in~\cite{Ararat2024} is the packing bound~\cite[Lemma~7.10]{Ararat2024}, and it gives}
\[
k \leq C_4 \cdot \left(\frac{R + \eta}{\varepsilon_k}\right)^{q-1}.
\]

\medskip
\noindent\emph{Step 4: Assembly.}
\revise{Substituting $\varepsilon_k = C_3\sqrt{\eta h_k}$ into the packing bound of Step~3 gives $k \leq C\, h_k^{-(q-1)/2}$, where the packing constant $C_4$ depends only on $q$ (it equals $q\pi_q/\pi_{q-1}$ in~\cite[Lemma~7.10]{Ararat2024}) and $C := C_4 (R+\eta)^{q-1} / (C_3^{\,q-1}\,\eta^{(q-1)/2})$ collects the constants of Steps~2--3, depending only on the norm, the dimension $q$, and the circumradius $R$ and inradius $\eta$ of $A$. Solving for $h_k$ gives}
\[
h_k \mathrel{\revise{\in}} O\bigl(k^{2/(1-q)}\bigr). \qedhere
\]
\end{proof}

\begin{remark}[Adaptive metrics]\label{rem:adaptive}
Theorem~\ref{thm:batch-rate} assumes a fixed $\ell_p$ norm. When the scalarization uses an iteration-dependent inner-product norm $\|z\|_{M_k}$ (the adaptive metric of~\cite{Alshahrani2026}), the Euclidean Hausdorff error still converges to zero provided the matrices $M_k$ satisfy uniform spectral bounds $mI \preceq M_k \preceq MI$, and the generic $H(r,A)$-sequence rate $O(k^{1/(1-q)})$ holds with \revise{$r = 1/\bigl(\sup_k \sqrt{\lambda_{\max}(M_k)/\lambda_{\min}(M_k)}\bigr)^3$}; see~\cite{Alshahrani2026}. Whether the improved exponent $2/(1-q)$ extends to iteration-dependent norms---either in the single-cut or batch setting---remains an open question: the deviation-vector separation argument of~\revise{\cite[Lemma~7.5]{Ararat2024}} requires geometric properties of the dual norm ball that change when the metric varies across iterations.
\end{remark}

\begin{remark}[Role of additional batch cuts]\label{rem:batch-improvement}
The rate $O(k^{2/(1-q)})$ in Theorem~\ref{thm:batch-rate} is stated in terms of outer iterations \revise{with the same exponent as} the standard algorithm~\revise{\cite[Theorem~7.2]{Ararat2024}}. However, batch cutting typically requires \emph{fewer} outer iterations to reach a given tolerance $\varepsilon$, because the additional $K - 1$ cuts per iteration further tighten the approximation: $h_{k+1} = \delta_H(P_{k+1}, A) \leq \delta_H(P_k \cap H(w^{(1)}_k, A), A)$, with the inequality potentially strict when any of the additional cuts are non-redundant. Each subsequent outer iteration then starts from a better approximation. This cumulative effect is confirmed by the iteration reductions reported in Section~\ref{sec:experiments}.
\end{remark}

We now analyze the combined effect of $\nw$-worker parallelism and $K$-batch cutting.

\begin{proposition}[Combined speedup]\label{prop:combined-speedup}
Let $S_\nw$ denote the parallel speedup from $\nw$ workers (equation~\eqref{eq:amdahl}), and let $\rho := k_{\mathrm{seq}} / k_{\mathrm{batch}} \geq 1$ denote the iteration reduction factor, where $k_{\mathrm{seq}}$ and $k_{\mathrm{batch}}$ are the iteration counts of the standard and batch algorithms to reach the same tolerance $\varepsilon$. If the average per-iteration cost were equal for both algorithms, the effective wall-clock speedup would be
\begin{equation}\label{eq:effective-speedup}
S_{\mathrm{eff}} = \rho \cdot S_\nw.
\end{equation}
\end{proposition}

\begin{remark}
The equal-cost assumption in Proposition~\ref{prop:combined-speedup} is only approximate. At a given iteration~$k$, both algorithms solve all $N_k$ subproblems, so the cost of iteration~$k$ itself is the same. Yet batch cutting adds $K$ halfspaces per iteration instead of one, so the vertex count $N_k$ grows faster and subsequent iterations are more expensive. The effective speedup $S_{\mathrm{eff}}$ is therefore generally less than $\rho \cdot S_\nw$. Our experiments show that $\rho \approx 2.7$--$5$ for $K = 5$, but the wall-clock speedup ranges from substantial gains on problems with moderate vertex growth \revise{(e.g., Ellipsoid3D and MOP7, with 73\% time savings)} to slowdowns on problems where the vertex count grows rapidly with each added cut \revise{(e.g., Ellipsoid4D and Distance); see Table~\ref{tab:batch}}.
\end{remark}

\section{Numerical experiments}\label{sec:experiments}

We now evaluate the parallel and batch-cutting algorithms on eight convex vector optimization problems. The experiments address four questions:
\begin{itemize}
\item \revise{how much speedup does parallelism deliver?}
\item \revise{how many iterations does batch cutting save, and does this translate to wall-clock gains?}
\item \revise{how does the batch size~$K$ affect the tradeoff between iteration reduction and per-iteration cost?}
\item \revise{how do these findings change with the objective dimension~$q$?}
\end{itemize}
We also compare the Euclidean norm against the adaptive metric~\cite{Alshahrani2026} and evaluate \revise{an LP-probe vertex strategy} as \revise{an alternative} to full enumeration.

\subsection{Implementation and hardware}\label{sec:setup}

All experiments were run on an Intel Core i9-9900K (8 physical cores, 3.60~GHz) with 32~GB RAM, using MATLAB R2025b with the Parallel Computing Toolbox. \revise{The parallel for-loop over vertices in Algorithm~\ref{alg:parallel} is implemented with MATLAB's \texttt{parfor} construct, which distributes the loop iterations across the worker pool.} Subproblems were solved using CVX~\cite{cvx} with the SDPT3 backend~\cite{Toh1999} at the default precision setting. Vertex enumeration used BENSOLVE Tools~\cite{bensolve}. The Hausdorff distance $\delta_H(P_k, A)$ was computed as $\max_{v \in \ext(P_k)} \|z^v\|$, which equals $\delta_H(P_k, A)$ by~\cite[Corollary~6.5]{Ararat2024}.

\subsection{Test problems}\label{sec:problems}

We test on eight problems from the CVOP literature, all with ordering cone $C = \R^q_+$. The problems span objective dimensions $q \in \{2,3,4,5\}$, decision space dimensions $n \in \{2,\ldots,6\}$, and constraint types ranging from spherical and ellipsoidal to polyhedral and box constraints. Table~\ref{tab:problems} summarizes the test suite.

\begin{example}[Ball~{\cite{Ararat2024}}]\label{ex:ball}
Minimize $\Gamma(x) = x$ with respect to $\leq_{\R^q_+}$ subject to $\|x - e\|_2 \leq 1$, where $e = (1,\ldots,1)^T \in \R^q$. The decision and objective spaces coincide ($n = q$), and the upper image boundary is a portion of the unit sphere centered at~$e$. This problem serves as the primary scaling benchmark since $q$ can be varied freely. We test with $q \in \{2, 3, 4, 5\}$.
\end{example}

\begin{example}[Distance~{\cite{Ararat2024}}]\label{ex:distance}
Let $a_1 = (1,1)^T$, $a_2 = (2,3)^T$, $a_3 = (4,2)^T \in \R^2$, and define
\[
\Gamma(x) = \bigl(\|x - a_1\|_2^2,\; \|x - a_2\|_2^2,\; \|x - a_3\|_2^2\bigr)^T.
\]
Minimize $\Gamma(x)$ with respect to $\leq_{\R^3_+}$ subject to $x_1 + 2x_2 \leq 10$, $0 \leq x_1 \leq 10$, $0 \leq x_2 \leq 4$. Here $n = 2$ and $q = 3$.
\end{example}

\begin{example}[Jahn~{\cite[Example~11.4]{Jahn2011}}]\label{ex:jahn}
Minimize $\Gamma(x) = (-x_1,\; x_1 + x_2^2)^T$ with respect to $\leq_{\R^2_+}$ subject to $x_1^2 - x_2 \leq 0$ and $x_1 + 2x_2 \leq 3$. Here $n = q = 2$. The nonlinear constraints and non-identity objective map distinguish this from the Ball problem.
\end{example}

\begin{example}[Ellipsoid~{\cite[Examples~5.3,~5.4]{Keskin2023}}]\label{ex:ellipsoid}
Minimize $\Gamma(x) = x$ with respect to $\leq_{\R^q_+}$ subject to
\[
\sum_{i=1}^{q} \left(\frac{x_i - 1}{d_i}\right)^2 \leq 1,
\]
where $d = (1, a, 5)^T$ for $q = 3$ and $d = (1, a, 5, 1)^T$ for $q = 4$. The semi-axis parameter~$a$ controls the eccentricity of the ellipsoid. We use $a = 10$ for $q = 3$ and $a = 5$ for $q = 4$, following~\cite{Keskin2023}. Like the Ball problem, the decision and objective spaces coincide ($n = q$), but the anisotropic boundary geometry tests the algorithm's behavior on non-spherical upper images. \revise{Concretely, the semi-axes differ by the factor~$a$, so the boundary curvature varies strongly with direction---in contrast to the sphere, whose curvature is uniform---and the support points, and hence the vertices of the outer approximation, concentrate in the high-curvature regions rather than distributing evenly. This stresses the vertex-enumeration and selection steps.}
\end{example}

\begin{example}[MOLP~{\cite[Example~5.5]{Keskin2023}}]\label{ex:molp}
Minimize $\Gamma(x) = Fx$ with respect to $\leq_{\R^q_+}$ subject to $Gx \leq b$ and $-100 \leq x_i \leq 100$, where $F \in \R^{q \times 2q}$, $G \in \R^{4q \times 2q}$, and $b \in \R^{4q}$ are generated randomly with a fixed seed for reproducibility ($F_{ij}, G_{ij} \sim \lceil 100 \cdot \mathcal{N}(0,1) \rceil$, $b_i \sim \lceil 10 \cdot \mathcal{U}(0,1) \rceil$). This multi-objective linear program tests scalability: the number of decision variables and constraints grows with~$q$. We test with $q = 3$.
\end{example}

\begin{example}[MOP7~{\cite{VIENNET1996}}]\label{ex:mop7}
Minimize $\Gamma(x) = (f_1(x), f_2(x), f_3(x))^T$ with respect to $\leq_{\R^3_+}$ subject to $-5 \leq x_i \leq 5$, $i = 1, 2$, where
\begin{align*}
f_1(x) &= \tfrac{1}{2}(x_1 - 2)^2 + \tfrac{1}{13}(x_2 + 1)^2 + 3, \\
f_2(x) &= \tfrac{1}{36}(x_1 + x_2 - 3)^2 + \tfrac{1}{8}(-x_1 + x_2 + 2)^2 - 17, \\
f_3(x) &= \tfrac{1}{175}(x_1 + 2x_2 - 1)^2 + \tfrac{1}{17}(-x_1 + 2x_2)^2 - 13.
\end{align*}
All three objectives are convex quadratic with $n = 2$, $q = 3$.
\end{example}

\begin{example}[AP1 \revise{(Ansary--Panda problem~1)}~{\cite[Example~1]{Ansary2015}}]\label{ex:ap1}
Minimize $\Gamma(x) = (f_1(x), f_2(x), f_3(x))^T$ with respect to $\leq_{\R^3_+}$ subject to $-5 \leq x_i \leq 5$, $i = 1, 2$, where
\begin{align*}
f_1(x) &= \tfrac{1}{4}\bigl[(x_1 - 1)^4 + 2(x_2 - 2)^4\bigr], \\
f_2(x) &= \exp\!\bigl(\tfrac{x_1 + x_2}{2}\bigr) + x_1^2 + x_2^2, \\
f_3(x) &= \tfrac{1}{6}\bigl[\exp(-x_1) + 2\exp(-x_2)\bigr].
\end{align*}
The objectives mix quartic, exponential, and negative-exponential terms---all convex---with $n = 2$, $q = 3$. The non-polynomial structure tests solver robustness.
\end{example}

\begin{table}[H]
\centering
\caption{Summary of test problems. \revise{Fifth column} ``Type'' indicates the structure of the objective map; \revise{last column} ``Constraints'' describes the feasible set.}\label{tab:problems}
\begin{tabular}{llccll}
\toprule
Problem & Source & \revise{\shortstack{Dimension of\\ objective space ($q$)}} & \revise{\shortstack{Dimension of\\ decision space ($n$)}} & Type & Constraints \\
\midrule
Ball & \cite{Ararat2024} & 2--5 & $q$ & Identity & Sphere \\
Distance & \cite{Ararat2024} & 3 & 2 & Quadratic & Polytope \\
Jahn & \cite{Jahn2011} & 2 & 2 & Nonlinear & Nonlinear + linear \\
Ellipsoid 3D & \cite{Keskin2023} & 3 & 3 & Identity & Ellipsoid ($a=10$) \\
Ellipsoid 4D & \cite{Keskin2023} & 4 & 4 & Identity & Ellipsoid ($a=5$) \\
MOLP & \cite{Keskin2023} & 3 & $2q$ & Linear & Polytope \\
MOP7 & \cite{VIENNET1996} & 3 & 2 & Quadratic & Box \\
AP1 & \cite{Ansary2015} & 3 & 2 & Quartic + exp & Box \\
\bottomrule
\end{tabular}
\end{table}

\subsection{Experimental setup}\label{sec:design}

\revise{The algorithm terminates once the Hausdorff approximation error satisfies $\delta_H(P_k, A) \leq \varepsilon$, where $\varepsilon > 0$ is the convergence tolerance.} This tolerance is set relative to the initial approximation error $\delta_0 := \delta_H(P_0, A)$, which varies by several orders of magnitude across problems (from $\delta_0 \approx 0.7$ for Ball $q\!=\!3$ to $\delta_0 \approx 18{,}000$ for MOLP). We set $\varepsilon = \varepsilon_{\mathrm{rel}} \cdot \delta_0$, where $\varepsilon_{\mathrm{rel}}$ accounts for the slower convergence at higher~$q$: $\varepsilon_{\mathrm{rel}} = 0.5\%$ for $q = 2$, $1\%$ for $q = 3$, $2\%$ for $q = 4$, and $5\%$ for $q = 5$.

We test six algorithm configurations, obtained by combining two scalarization metrics \revise{(Euclidean and adaptive~\cite{Alshahrani2026})}, \revise{each} with three vertex strategies \revise{(full enumeration, LP probes, and LP probes with periodic full enumeration)}, as summarized in Table~\ref{tab:configs}. The Euclidean norm uses the fixed $\|z\|_2$ in the subproblem~\eqref{eq:subproblem}. The adaptive metric~\cite{Alshahrani2026} replaces $\|z\|_2$ with an iteration-dependent inner-product norm $\|z\|_{M_k} := \sqrt{z^T M_k z}$, where $M_k := \varepsilon_0 I + \frac{1}{k}\sum_{i=1}^{k} u_i u_i^T$ is built from the accumulated cut normals $u_1, \ldots, u_k$ and a regularization parameter $\varepsilon_0 > 0$. The changing metric adapts to the local geometry of the upper image boundary, \revise{improving the multiplicative constant in the convergence bound---which decreases as the cut-normals disperse, i.e., as $r = 1/\bigl(\sup_k \sqrt{\lambda_{\max}(M_k)/\lambda_{\min}(M_k)}\bigr)^3$ increases toward~$1$---while preserving the generic rate exponent $1/(1-q)$~\cite[Theorem~18]{Alshahrani2026}}. The approximation error is always measured in the fixed Euclidean norm; only the scalarization subproblem uses $M_k$. Three vertex-finding strategies are compared. The \emph{full enumeration} strategy computes $\ext(P_k)$ at every iteration and evaluates all vertices. The \emph{LP probe} strategy avoids full enumeration: instead of computing $\ext(P_k)$, it solves a small number of LPs over $P_k$ in selected probe directions (the standard basis vectors $e_1, \ldots, e_q$ and the accumulated cut normals $w^1, \ldots, w^{k-1}$). Each LP returns a vertex of $P_k$ that is extremal in that direction and likely far from $A$, so a subset of informative vertices is identified without enumerating the full vertex set. The \emph{hybrid} strategy alternates LP probes with periodic full enumeration (every 50 iterations) to guard against missing vertices that no probe direction reaches.

\begin{table}[H]
\centering
\caption{Six algorithm configurations combining two scalarization metrics with three vertex-finding strategies.}\label{tab:configs}
\begin{tabular}{lll}
\toprule
Configuration & Scalarization metric & Vertex strategy \\
\midrule
Euclid+Full & Euclidean ($\ell_2$) & Full enumeration \\
Adapt+Full & Adaptive~\cite{Alshahrani2026} & Full enumeration \\
Euclid+Hybrid & Euclidean ($\ell_2$) & LP probes + periodic full \\
Euclid+LP & Euclidean ($\ell_2$) & LP probes only \\
Adapt+Hybrid & Adaptive~\cite{Alshahrani2026} & LP probes + periodic full \\
Adapt+LP & Adaptive~\cite{Alshahrani2026} & LP probes only \\
\bottomrule
\end{tabular}
\end{table}

\subsection{Results}\label{sec:results}

We begin by comparing the six configurations in sequential mode to identify which norm--strategy combinations are most effective, then turn to parallelism, batch cutting, batch size sensitivity, and scaling with~$q$.

\subsubsection{\revise{Sequential performance and vertex-finding strategies}}

Table~\ref{tab:config-full} compares the two full enumeration configurations across all eight test problems. \revise{Solving the subproblems accounts for over 99\%} of the computation time in all cases, confirming that the subproblem evaluation step is the bottleneck. The empirical convergence slopes are consistent with or steeper than the theoretical rate $O(k^{2/(1-q)})$: slopes near $-1$ for $q = 3$ and near $-2$ for $q = 2$.

\begin{table}[H]
\centering
\caption{Sequential performance with full vertex enumeration, comparing the Euclidean norm against the adaptive metric~\cite{Alshahrani2026}. Slope is the empirical convergence rate from a log-log fit of the Hausdorff error versus iteration count; the theoretical exponent is $2/(1-q)$.}\label{tab:config-full}
\begin{tabular}{lcrrrrrr}
\toprule
 & & \multicolumn{3}{c}{Euclid+Full} & \multicolumn{3}{c}{Adapt+Full} \\
\cmidrule(lr){3-5} \cmidrule(lr){6-8}
Problem & \revise{$q$} & Iters & Time (s) & Slope & Iters & Time (s) & Slope \\
\midrule
Ball & \revise{3} & 94 & 373 & $-1.16$ & 102 & 326 & $-1.12$ \\
Distance & \revise{3} & 48 & 70 & $-1.36$ & 51 & 143 & $-1.33$ \\
Jahn & \revise{2} & 14 & 13 & $-1.81$ & 24 & 18 & $-1.48$ \\
Ellipsoid3D & \revise{3} & 53 & 357 & $-1.18$ & 50 & 279 & $-1.22$ \\
Ellipsoid4D & \revise{4} & 63 & 479 & $-1.07$ & 73 & 309 & $-1.01$ \\
MOLP & \revise{3} & 34 & 123 & $-1.60$ & \multicolumn{3}{c}{failed\textsuperscript{$\dagger$}} \\
MOP7 & \revise{3} & 15 & 24 & $-1.42$ & 23 & 44 & $-1.51$ \\
AP1 & \revise{3} & 33 & 413 & $-1.28$ & 29 & 260 & $-1.40$ \\
\bottomrule
\end{tabular}

\smallskip\noindent\textsuperscript{$\dagger$}\revise{Here ``failed'' means that the run did not complete: the iteration-dependent (adaptive) metric repeatedly produced ill-conditioned outer-approximation polytopes on which the vertex-enumeration routine broke down, so no result is reported for this configuration.}
\end{table}

Figure~\ref{fig:convergence} shows the convergence behavior on four representative problems; the empirical error decay closely follows the theoretical rate from Table~\ref{tab:config-full}.

\begin{figure}[H]
\centering
\includegraphics[width=\textwidth]{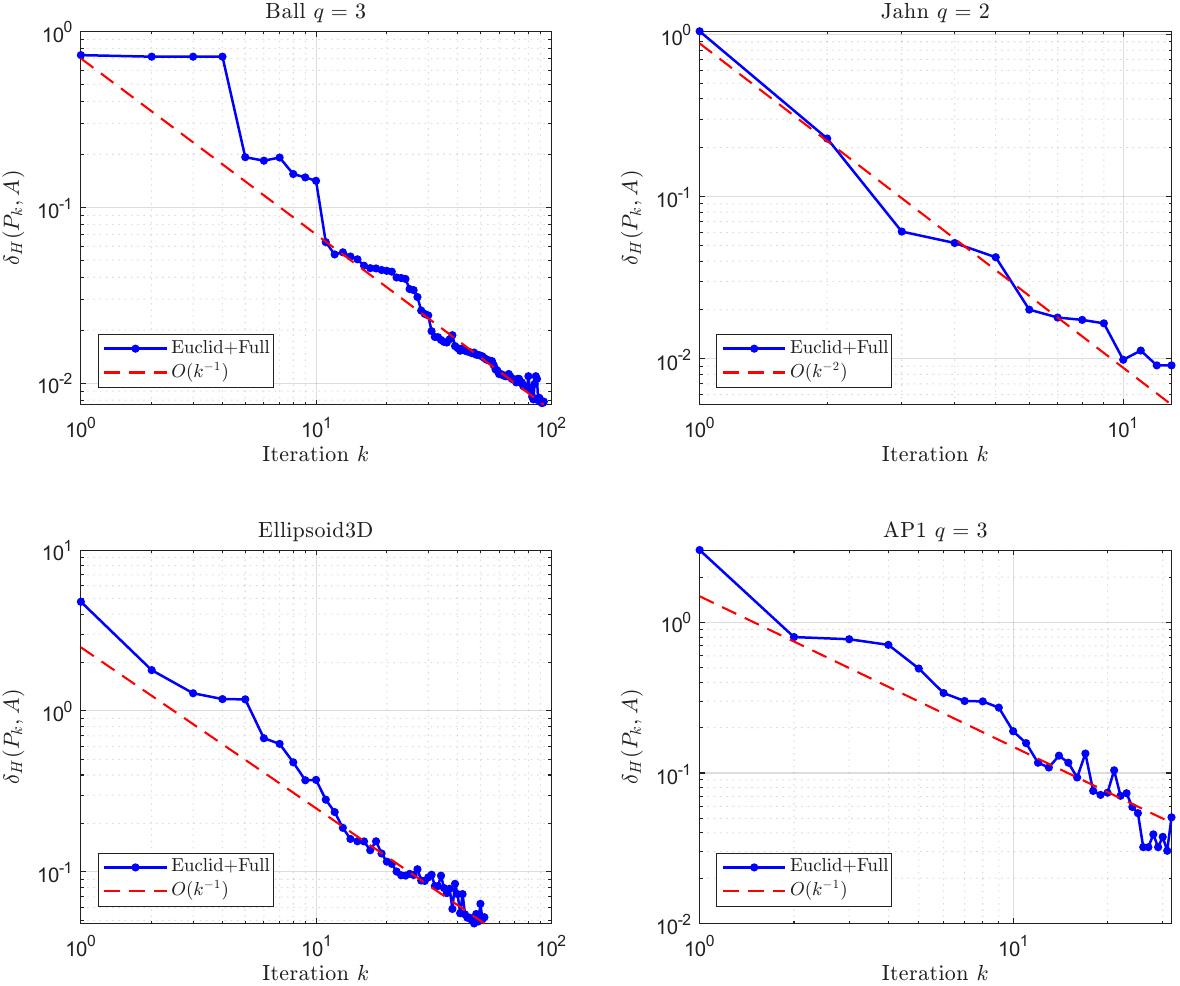}
\caption{Convergence of the Euclid+Full configuration on four test problems (log-log scale). The dashed line shows the theoretical rate $O(k^{2/(1-q)})$: exponent $-1$ for $q = 3$ (Ball, Ellipsoid3D, AP1) and $-2$ for $q = 2$ (Jahn). \revise{By nestedness of the outer approximations ($P_0 \supseteq P_1 \supseteq \cdots \supseteq A$), $\delta_H(P_k, A)$ is non-increasing in $k$; the small non-monotone fluctuations visible near convergence are numerical artifacts of evaluating the per-vertex subproblems at the solver's default precision, not a property of the algorithm.}}\label{fig:convergence}
\end{figure}

The adaptive metric from~\cite{Alshahrani2026} yields mixed results: it reduces wall-clock time on Ellipsoid3D (279\,s vs.\ 357\,s) and AP1 (260\,s vs.\ 413\,s), but increases it on Distance (143\,s vs.\ 70\,s) and Jahn (18\,s vs.\ 13\,s), and fails entirely on MOLP due to numerical instability. Since the metric changes at each iteration, solution caches must use a norm-independent criterion to avoid \revise{stale vertex classifications---cached labels recording whether each vertex's distance to $A$ is within tolerance, which become invalid when the metric (and hence the measured distance $\|z^v\|_{M_k}$) changes, so that a vertex cached as converged under one metric may no longer be converged under the next}; in our implementation we use Euclidean distance for this purpose.

The vertex-finding strategy has a larger effect on performance than the norm choice. Table~\ref{tab:config-strategies} compares the LP probe and hybrid strategies against the full enumeration baseline from Table~\ref{tab:config-full}.

\begin{table}[H]
\centering
\caption{Effect of the vertex-finding strategy on sequential performance (Euclidean norm). \revise{Rather than enumerating all vertices of $P_k$, the LP probe and hybrid strategies find a subset of vertices by solving one linear program over $P_k$ per probe direction.} The full enumeration column repeats the Euclid+Full data from Table~\ref{tab:config-full} for comparison.}\label{tab:config-strategies}
\begin{tabular}{lcrrrrrr}
\toprule
 & & \multicolumn{2}{c}{Euclid+Hybrid} & \multicolumn{2}{c}{Euclid+LP} & \multicolumn{2}{c}{Euclid+Full} \\
\cmidrule(lr){3-4} \cmidrule(lr){5-6} \cmidrule(lr){7-8}
Problem & \revise{$q$} & Iters & Time (s) & Iters & Time (s) & Iters & Time (s) \\
\midrule
Ball & \revise{3} & 36 & 19 & 34 & 17 & 94 & 373 \\
Distance & \revise{3} & 17 & 11 & 17 & 11 & 48 & 70 \\
Jahn & \revise{2} & 10 & 5 & 8 & 4 & 14 & 13 \\
Ellipsoid3D & \revise{3} & 22 & 14 & 22 & 13 & 53 & 357 \\
Ellipsoid4D & \revise{4} & 34 & 24 & 34 & 22 & 63 & 479 \\
MOLP & \revise{3} & 19 & 10 & 18 & 9 & 34 & 123 \\
MOP7 & \revise{3} & 13 & 12 & 13 & 11 & 15 & 24 \\
AP1 & \revise{3} & 20 & 34 & 20 & 33 & 33 & 413 \\
\bottomrule
\end{tabular}
\end{table}

\revise{The LP-based strategies achieve 10--25$\times$ reductions in wall-clock time relative to full enumeration on Ball, Ellipsoid3D, Ellipsoid4D, and AP1, along with a 40--60\% decrease in iteration count. The hybrid and LP probe strategies exhibit near-identical performance, indicating that periodic full enumeration incurs negligible computational overhead.} Combining the adaptive metric with LP-based strategies (not shown) yields similar timings, with no clear advantage over the Euclidean norm. These strategies are orthogonal to the parallel and batch techniques analyzed next: \revise{LP probes reduce the per-iteration cost by evaluating fewer vertices, batch cutting reduces the number of outer iterations, and parallelism reduces the per-iteration wall-clock time by distributing the $N_k$ independent subproblems across workers. The synchronous parallel implementation evaluated below (Algorithm~\ref{alg:parallel}) uses full vertex enumeration, and therefore does not yet combine parallelism with the LP-probe strategy.}

\subsubsection{\revise{How much speedup does parallelism deliver?}}

\revise{Because solving the subproblems accounts for almost all of the per-iteration computation and the $N_k$ subproblems at each iteration are mutually independent---each depends only on its vertex and the target set $A$---they are a natural target for parallelization.} Table~\ref{tab:parallel} compares synchronous parallel execution on 8~workers against the sequential Euclid+Full baseline from Table~\ref{tab:config-full}.

\begin{table}[H]
\centering
\caption{Synchronous parallel speedup on 8 workers, compared against the sequential Euclid+Full baseline from Table~\ref{tab:config-full}. Both use full vertex enumeration; the only difference is that \revise{the subproblems are solved concurrently across workers via \texttt{parfor}}.}\label{tab:parallel}
\begin{tabular}{lcrrrrr}
\toprule
 & & \multicolumn{2}{c}{Sequential} & \multicolumn{2}{c}{Parallel (8 workers)} & \\
\cmidrule(lr){3-4} \cmidrule(lr){5-6}
Problem & \revise{$q$} & Iters & Time (s) & Iters & Time (s) & Speedup \\
\midrule
Ball & \revise{3} & 94 & 373 & 94 & 155 & 2.4$\times$ \\
Distance & \revise{3} & 48 & 70 & 48 & 31 & 2.3$\times$ \\
Jahn & \revise{2} & 14 & 13 & 17 & 11 & 1.2$\times$ \\
Ellipsoid3D & \revise{3} & 53 & 357 & 55 & 144 & 2.5$\times$ \\
Ellipsoid4D & \revise{4} & 63 & 479 & 64 & 167 & 2.9$\times$ \\
MOLP & \revise{3} & 34 & 123 & 34 & 48 & 2.6$\times$ \\
MOP7 & \revise{3} & 15 & 24 & 15 & 22 & 1.1$\times$ \\
AP1 & \revise{3} & 33 & 413 & 27 & 99 & 4.2$\times$ \\
\bottomrule
\end{tabular}
\end{table}

Speedups range from $1.1\times$ (MOP7, where few vertices limit parallelism) to $4.2\times$ (AP1, \revise{where the subproblems are expensive to solve}). Iteration counts are nearly identical in most cases; minor differences (e.g., Jahn: 14 vs.\ 17 iterations, AP1: 33 vs.\ 27) arise from floating-point nondeterminism in the parallel solver, which can produce slightly different cut sequences.

\subsubsection{\revise{How many iterations does batch cutting save, and does this translate to wall-clock gains?}}

While parallelism reduces per-iteration cost, the algorithm still adds only a single cut per iteration, discarding the solved subproblem information for all other vertices. Batch cutting addresses this by adding up to $K$ cuts per iteration. Table~\ref{tab:batch} compares batch cutting ($K = 5$, 8~workers) against synchronous parallel.

\begin{table}[H]
\centering
\caption{Batch cutting ($K = 5$~\revise{ with} 8 workers) versus synchronous parallel \revise{($K = 1$ with 8 workers)}. Iteration and time reductions are relative to synchronous parallel. Negative time reduction indicates batch is slower in wall-clock time despite fewer iterations.}\label{tab:batch}
\begin{tabular}{lcrrrrrr}
\toprule
 & & \multicolumn{2}{c}{Sync} & \multicolumn{2}{c}{Batch ($K\!=\!5$)} & Iter. & Time \\
\cmidrule(lr){3-4} \cmidrule(lr){5-6}
Problem & \revise{$q$} & Iters & Time (s) & Iters & Time (s) & red. & red. \\
\midrule
Ball & \revise{3} & 94 & 155 & 24 & 109 & 74\% & 30\% \\
Distance & \revise{3} & 48 & 31 & 14 & 45 & 71\% & $-$45\%\textsuperscript{$*$} \\
Jahn & \revise{2} & 17 & 11 & 6 & 5 & 65\% & 55\% \\
Ellipsoid3D & \revise{3} & 55 & 144 & 13 & 39 & 76\% & 73\% \\
Ellipsoid4D & \revise{4} & 64 & 167 & 24 & 366 & 63\% & $-$119\%\textsuperscript{$*$} \\
MOLP & \revise{3} & 34 & 48 & \multicolumn{4}{c}{failed\textsuperscript{$\dagger$}} \\
MOP7 & \revise{3} & 15 & 22 & 3 & 6 & 80\% & 73\% \\
AP1 & \revise{3} & 27 & 99 & 9 & 44 & 67\% & 56\% \\
\bottomrule
\end{tabular}

\smallskip\noindent\textsuperscript{$*$}\revise{On these problems, the additional cuts accelerate the growth of the vertex count, so that the increase in per-iteration cost outweighs the reduction in iteration count, resulting in a net increase in wall-clock time.}

\smallskip\noindent\textsuperscript{$\dagger$}\revise{In batch mode, the outer-approximation polytope for the MOLP problem became numerically ill-conditioned, causing the convex solver to stall; the run did not complete, so no result is reported.}
\end{table}

Batch cutting consistently reduces iterations by 62--80\%, with MOP7 achieving the largest reduction (from 15 to 3 iterations). \revise{The wall-clock performance, however, presents a more nuanced picture.} On problems where the per-iteration cost grows modestly with the number of cuts---Ball, Ellipsoid3D, MOP7, AP1---batch cutting \revise{yields wall-clock time reductions of 30--73\%}. On Distance and Ellipsoid4D, the additional cuts accelerate vertex count growth, \revise{rendering} each batch iteration more expensive than several synchronous ones. This is the central tradeoff of batch cutting: each \revise{additional} cut tightens the \revise{polyhedral} approximation, \revise{but simultaneously} increases the polytope complexity that subsequent iterations must process.

\subsubsection{\revise{How does the batch size $K$ affect the tradeoff between iteration reduction and per-iteration cost?}}

The batch size~$K$ controls this tradeoff. Table~\ref{tab:batch-k} varies $K$ from 1 (one cut per iteration, equivalent to synchronous parallel) to \revise{$N^\varepsilon_k$ (all eligible cuts, i.e., all $v$ with $\|z^v\| > \varepsilon$)} on four representative problems.

\begin{table}[H]
\centering
\caption{Effect of batch size $K$ on iteration count and wall-clock time (Euclidean norm, 8 workers). The fastest wall-clock time for each problem is underlined.}\label{tab:batch-k}
\begin{tabular}{lrrrrr}
\toprule
\revise{Batch size $K$} & \revise{1} & \revise{3} & \revise{5} & \revise{10} & \revise{$N^\varepsilon_k$} \\
\midrule
\multicolumn{6}{l}{\emph{Iterations}} \\
Ball $q\!=\!3$ & 94 & 41 & 27 & 14 & 10 \\
Ellipsoid3D & 49 & 21 & 15 & 9 & 8 \\
MOP7 $q\!=\!3$ & 17 & 7 & 3 & 3 & 3 \\
AP1 $q\!=\!3$ & 24 & 12 & 8 & 7 & 6 \\
\midrule
\multicolumn{6}{l}{\emph{Wall-clock time (s)}} \\
Ball $q\!=\!3$ & 145 & 159 & 129 & 61 & \underline{45} \\
Ellipsoid3D & 100 & 57 & 46 & \underline{26} & 27 \\
MOP7 $q\!=\!3$ & 10 & 7 & \underline{3} & 3 & 3 \\
AP1 $q\!=\!3$ & 75 & 45 & \underline{30} & 33 & 31 \\
\bottomrule
\end{tabular}
\end{table}

Increasing $K$ consistently reduces iterations, with diminishing returns beyond $K = 5$; MOP7 saturates at 3~iterations for all $K \geq 5$. The optimal $K$ for wall-clock time \revise{is problem dependent}: \revise{$K = N^\varepsilon_k$} is fastest for Ball (45\,s), $K = 10$ for Ellipsoid3D (26\,s), and $K = 5$ for AP1 (30\,s). A moderate choice of $K \in \{5, 10\}$ provides a reasonable balance across all tested problems. Figure~\ref{fig:batch-k} visualizes this tradeoff: the left panel shows monotone iteration reduction with diminishing returns, while the right panel reveals the non-monotone wall-clock behavior---for Ball, time initially increases at $K = 3$ before decreasing, reflecting the cost of processing more vertices per iteration.

\begin{figure}[H]
\centering
\includegraphics[width=\textwidth]{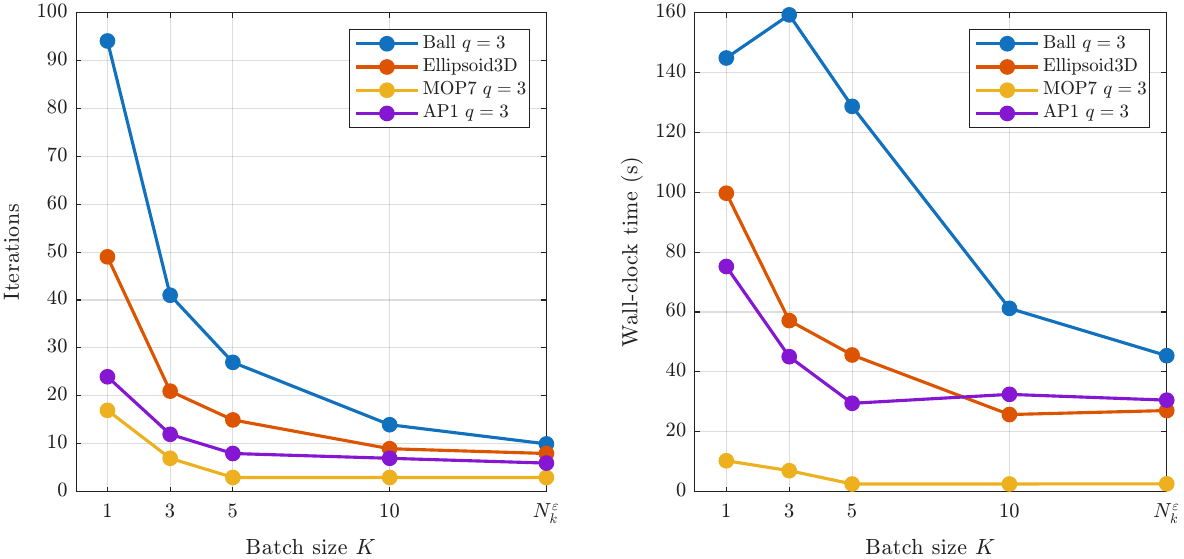}
\caption{Effect of batch size $K$ on iteration count (left) and wall-clock time (right) for four test problems (Euclidean norm, 8 workers). \revise{The total number of iterations decreases} monotonically with $K$, but wall-clock time is non-monotone: the per-iteration cost \revise{due to} additional cuts can outweigh the iteration savings at intermediate~$K$.}\label{fig:batch-k}
\end{figure}

\subsubsection{\revise{How do these findings change with the objective dimension $q$?}}

Finally, Table~\ref{tab:scaling} examines how these findings scale with the objective dimension~$q$, comparing synchronous and \revise{batch parallel (the batch-cutting algorithm of Algorithm~\ref{alg:batch}, here with $K = 5$, run on the same $\nw = 8$ workers)} on the Ball problem across $q \in \{2, 3, 4, 5\}$.

\begin{table}[H]
\centering
\caption{Scaling of synchronous and batch parallel ($K = 5$) with objective dimension~$q$ on the Ball problem (8 workers, Euclidean norm). Iteration and time reductions are relative to synchronous parallel.}\label{tab:scaling}
\begin{tabular}{crrrrrr}
\toprule
 & \multicolumn{2}{c}{Sync} & \multicolumn{2}{c}{Batch ($K\!=\!5$)} & Iter. & Time \\
\cmidrule(lr){2-3} \cmidrule(lr){4-5}
$q$ & Iters & Time (s) & Iters & Time (s) & red. & red. \\
\midrule
2 & 21 & 16 & 8 & 9 & 62\% & 44\% \\
3 & 92 & 127 & 30 & 152 & 67\% & $-$20\% \\
4 & 113 & 342 & 41 & 1003 & 64\% & $-$193\% \\
5 & 64 & 638 & 21 & 788 & 67\% & $-$24\% \\
\bottomrule
\end{tabular}
\end{table}

The iteration reduction is consistent at 62--67\% across all values of~$q$, confirming that batch cutting is an effective iteration-reduction mechanism independent of problem dimension. The wall-clock tradeoff, however, worsens at higher~$q$: at $q = 4$, batch takes nearly three times as long as synchronous parallel despite using only 36\% of the iterations, because the vertex count grows \revise{rapidly} with~$q$ and each additional batch cut amplifies this growth. \revise{Batch cutting is therefore most effective when solving the subproblems dominates the per-iteration cost and the vertex count stays moderate---so that the additional cuts add little overhead---or when the tolerance is tight enough that the number of iterations, rather than the per-iteration cost, becomes the primary bottleneck.}

The LP probe strategy offers the largest sequential speedups (10--25$\times$) by avoiding full vertex enumeration entirely; integrating it with the parallel and batch framework is a promising direction.

\section{Conclusion}\label{sec:conclusion}

We \revise{introduce} parallel and batch-cutting variants of the norm-minimization-based CVOP algorithm and \revise{prove} that batch cutting preserves the convergence rate $O(k^{2/(1-q)})$ (Theorem~\ref{thm:batch-rate}). The key insight is that each outer iteration of the batch algorithm still includes the \revise{cut corresponding to the farthest vertex}, so the deviation-vector separation argument of~\revise{\cite[Lemma~7.5]{Ararat2024} still holds}.

On the computational side, synchronous parallelism \revise{increases the speed by a factor of 1.1 to 4.2} on 8~cores, and batch cutting at $K = 5$ consistently reduces the iteration count by 62--80\%. The wall-clock benefit of batch cutting, however, depends on the problem: \revise{each additional cut tightens the approximation but also enlarges the outer-approximation polytope, so on problems where the vertex count grows quickly the higher per-iteration cost can outweigh the savings from fewer iterations. This effect becomes more pronounced as the objective dimension~$q$ increases, and is governed by the batch size~$K$: among the values tested, a moderate choice $K \in \{5, 10\}$ gave the best balance between the reduction in iteration count and the added per-iteration overhead.}

Several directions remain open:
\begin{enumerate}[label=(\arabic*)]
\item Replacing the general-purpose convex solver (CVX) with a specialized solver exploiting the shared constraint structure of the subproblems could \revise{substantially reduce the time to solve each subproblem}. Since all $N_k$ subproblems share the same feasible set and differ only in the reference point $v$, \revise{two strategies are natural candidates: \emph{warm-starting}---initializing each subproblem from the optimal solution of a previously solved, nearby one to reduce the number of solver iterations---and \emph{batched} solvers, which solve many of these similar subproblems together to amortize the per-call setup overhead}.

\item GPU-accelerated solvers for batched convex programs\revise{---e.g., quadratic programs (QPs) or second-order cone programs (SOCPs)---}could further exploit the structure, particularly for large~$q$ where $N_k$ grows.

\item The convergence rate $O(k^{2/(1-q)})$ in Theorem~\ref{thm:batch-rate} is stated in terms of outer iterations. A tighter analysis that accounts for all $k_{\mathrm{eff}} = Kk$ cuts \revise{accumulated after $k$ iterations}---not just the \revise{$k$} farthest-vertex cuts---could potentially yield the sharper bound $O(k_{\mathrm{eff}}^{\,2/(1-q)})$. This would require extending the deviation vector separation argument of~\cite{Ararat2024} to non-farthest-vertex cuts within a batch, which we leave as an open problem.

\item \revise{Vertex enumeration is the main step of each iteration that is not parallelized. Instead of recomputing $\ext(P_{k+1})$ from scratch, one could update it \emph{incrementally} when a halfspace is added: only the vertices that the new hyperplane cuts off are removed and the new vertices it creates are added, while the rest of $\ext(P_k)$ is left unchanged. This would avoid full enumeration at every iteration and reduce the serial fraction in Amdahl's law further, improving the parallel speedup; for batch cutting, the update would be applied for each of the $K$ halfspaces added in an iteration.}

\item \revise{As our experiments indicate, the growth in per-iteration cost induced by batch cutting can offset its reduction in the iteration count, particularly at higher objective dimensions~$q$. Combining batch cutting with the LP-probe strategy may alleviate this tradeoff, since the latter restricts the number of vertices evaluated per iteration and thereby curbs the growth in per-iteration cost.}
\end{enumerate}

\section*{Declarations}

\textbf{Conflict of interest.} The author declares no conflict of interest.

\textbf{Data availability.} All numerical results were generated by the algorithms described in this paper. No external datasets were used. The code is available from the author upon reasonable request.

\textbf{Funding.} This research did not receive any specific grant from funding agencies in the public, commercial, or not-for-profit sectors.

\textbf{AI use.} During the preparation of this work, the author used Claude (Anthropic) to assist with manuscript editing, including tightening prose, verifying LaTeX formatting, and checking internal consistency of cross-references and notation. The author reviewed and edited all output and takes full responsibility for the content of the publication.

\section*{Acknowledgment}
The author acknowledges the institutional support provided by King Fahd University of Petroleum \& Minerals (KFUPM) and the Interdisciplinary Research Center for Smart Mobility and Logistics (IRC-SML) at KFUPM.

\printbibliography

\end{document}